\documentclass[runningheads,envcountsame]{llncs}
\usepackage{amssymb,amsmath}
\usepackage{graphicx}
\DeclareGraphicsExtensions{.pdf,.eps}
\graphicspath{{figures/}}
\usepackage{enumitem}
\usepackage{microtype}

\let\doendproof\endproof
\renewcommand\endproof{\hfill\qed\doendproof}

\usepackage[
    hidelinks,
    pdfauthor={Laura Merker and Torsten Ueckerdt},
    pdftitle={The Local Queue Number of Graphs with Bounded Treewidth},
    pdfsubject={Appears in the Proceedings of the 28th International Symposium on Graph Drawing and Network Visualization (GD 2020)},
    pdfkeywords={queue Number, local covering number, treewidth}
]{hyperref}
\usepackage[noabbrev, capitalise]{cleveref}
\crefname{cond}{Condition}{Conditions}
\usepackage{cite}
\usepackage{xspace}
\newcommand{\calG}{\ensuremath{\mathcal{G}}\xspace}
\newcommand{\calQ}{\ensuremath{\mathcal{Q}}\xspace}

\DeclareMathOperator{\mad}{mad}
\DeclareMathOperator{\spans}{span}
\DeclareMathOperator{\qn}{qn}
\DeclareMathOperator{\pn}{pn}
\DeclareMathOperator{\uqn}{qn_u}

\DeclareMathOperator{\lqn}{qn_\ell}
\DeclareMathOperator{\lpn}{pn_\ell}
\DeclareMathOperator{\arb}{arb}
\DeclareMathOperator{\sa}{sa}
\renewcommand{\preceq}{\preccurlyeq}

\renewcommand{\leq}{\leqslant}
\renewcommand{\geq}{\geqslant}

\newcommand{\game}[1]{\text{Game}~\labelcref{#1}\xspace}
\newcommand{\gameRed}[2]{\subsubsection{\game{#1} $ \boldsymbol{\rightsquigarrow} $ \game{#2}.}}
\newcommand{\cinit}{\ensuremath{C_\text{init}}\xspace}

\newcommand{\nameColorA}{black\xspace}
\newcommand{\nameColorB}{orange\xspace}
\newcommand{\nameColorC}{blue\xspace}
\DeclareMathOperator{\opColorA}{\nameColorA}
\DeclareMathOperator{\opColorB}{\nameColorB}
\DeclareMathOperator{\opColorC}{\nameColorC}
\newcommand{\qa}{\ensuremath{Q_{\opColorA}}\xspace}
\newcommand{\qb}{\ensuremath{Q_{\opColorB}}\xspace}
\newcommand{\qc}{\ensuremath{Q_{\opColorC}}\xspace}

\begin{document}
\title{The Local Queue Number of Graphs with Bounded Treewidth}
\author{Laura Merker \and Torsten Ueckerdt}
\authorrunning{L.~Merker and T.~Ueckerdt}
\institute{
    Karlsruhe Institute of Technology (KIT), Institute of Theoretical Informatics, Germany\\
    \email{laura.merker@student.kit.edu, torsten.ueckerdt@kit.edu}
}
\maketitle 
\begin{abstract}
    A queue layout of a graph $G$ consists of a vertex ordering of $G$ and a partition of the edges into so-called queues such that no two edges in the same queue nest, i.e., have their endpoints ordered in an ABBA-pattern.
    Continuing the research on local ordered covering numbers, we introduce the local queue number of a graph $G$ as the minimum $\ell$ such that $G$ admits a queue layout with each vertex having incident edges in no more than $\ell$ queues.
    Similarly to the local page number [Merker, Ueckerdt, GD'19], the local queue number is closely related to the graph's density and can be arbitrarily far from the classical queue number.

    We present tools to bound the local queue number of graphs from above and below, focusing on graphs of treewidth $k$.
    Using these, we show that every graph of treewidth $k$ has local queue number at most $k+1$ and that this bound is tight for $k=2$, while a general lower bound is $\lceil k/2\rceil+1$.
    Our results imply, inter alia, that the maximum local queue number among planar graphs is either 3 or 4.
    
\end{abstract}

\section{Introduction}
\label{sec:intro}

Given a graph, we aim to find a vertex ordering $ \prec $ and a partition of the edges into queues, where two edges $ uv $ and $ xy $ may not be in the same queue if $ u \prec x \prec y \prec v $.
Since Heath and Rosenberg~\cite{1queue} introduced this concept in 1992, one of the main concerns of studying queue layouts is the investigation of the maximum queue number of the class of planar graphs and the class of graphs with bounded treewidth, see for instance~\cite{treewidth,treewidthExp,2trees,planar3TreeQN,compQS,planarBoundedQN}.
Despite recent breakthroughs, there are still large gaps between lower and upper bounds on the maximum queue number of both graph classes.
In particular, the maximum queue number of planar graphs is between $ 4 $ and $ 49 $ due to Alam et~al.~\cite{planar3TreeQN}, respectively Dujmovi\'c et~al.~\cite{planarBoundedQN}, and Wiechert~\cite{treewidthExp} provides a linear lower bound and an exponential upper bound on the maximum queue number of graphs with treewidth $ k $.
We continue the research in this direction by proposing a new graph parameter, the \emph{local queue number}, that minimizes the number of queues in which any one vertex has incident edges.
Compared to the classical queue number, the investigation of the local queue number leads to stronger lower bounds and weaker upper bounds.
The latter might offer a way to support conjectured upper bounds on the classical queue number.
We remark that analogously to the local queue number considered here, we recently introduced~\cite{localPageNumbers} the \emph{local page number} as a weaker version of the classical page number.

All necessary definitions are given in \cref{sec:definitions}, including the formal definition of local queue numbers.
In \cref{sec:relatedWork}, we briefly locate local queue numbers in the general covering number framework, and outline the state of the art on queue numbers and local page numbers of planar graphs and graphs with bounded treewidth.
We summarize our results in \cref{sec:contribution} and point out which results on local page numbers immediately generalize to local queue numbers.
We then investigate the local queue number of $ k $-trees in \cref{sec:ktrees}. 
Finally, we discuss possible applications of the presented tools and propose open problems for further research in \cref{sec:conclusions}.

\subsection{Definitions}
\label{sec:definitions}

Consider a graph $ G $ with a linear ordering $ \prec $ of its vertex set.
The sets $ V(G) $ and $ E(G) $ denote the vertex set, respectively edge set, of $ G $.
For subsets $ X, Y \subseteq V(G)$, we write $ X \prec Y $ and say $ X $ is \emph{to the left} of $ Y $ and $ Y $ is \emph{to the right} of $ X $ if $ x \prec y $ for all vertices $ x \in X, y \in Y $.
If the sets consist only of a single vertex, we use $ x $ instead of $ \{ x \} $.
Let the \emph{span} of $ X $ contain all vertices lying between the leftmost and the rightmost vertex of $ X $, that is $ \spans(X) = \{v \in V(G) \colon \exists \, x, x' \in X \text{ with } x \preceq v \preceq x'\} $.
For a subgraph $ H $ of $ G $ and a vertex $ v \not\in V(H) $, we say $ v $ is \emph{below} $ H $ if $ v \in \spans(V(H)) $ and we say $ v $ is \emph{outside} $ H $ otherwise.

\begin{figure}
    \begin{center}
        \includegraphics{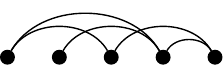}\hfill
        \includegraphics{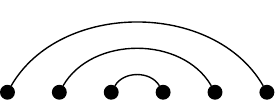}\hfill
        \includegraphics{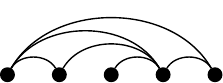}\hfill
        \includegraphics{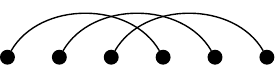}
    \end{center}
    \caption{Left to right: $ 1 $-queue layout, $ 3 $-rainbow, $ 1 $-page book embedding, $ 3 $-twist.}
    \label{fig:queue}
    \label{fig:page}
\end{figure}

Two edges $ uv, xy \in E(G) $ \emph{nest} if $ u \prec x \prec y \prec v $ or $ x \prec u \prec v \prec y $, and they \emph{cross} if $ u \prec x \prec v \prec y $ or $ x \prec u \prec y \prec v $.
A set of $ k $ pairwise nesting (crossing) edges is called a \emph{$ k $-rainbow} (\emph{$ k $-twist}).
A \emph{queue} (\emph{page}) is an edge set in which no two edges nest (cross), see \cref{fig:queue,fig:page}.
A \emph{$ k $-queue layout} (\emph{$ k $-page book embedding}) of $ G $ consists of a vertex ordering $ \prec $ and a partition of the edges of $ G $ into $ k $ queues (pages).
Finally, the \emph{queue number} $ \qn(G) $ (\emph{page number} $ \pn(G) $, also known as \emph{stack number} or \emph{book thickness}) of a graph $ G $ is the smallest $ k $ such that there is a $ k $-queue layout ($ k $-page book embedding) for $ G $.
Both concepts are called \emph{ordered covering numbers} as a partition of edges can also be considered as covering the graph with queues or pages, respectively.

We now define local variants of the parameters defined above.
For this, we allow partitions of arbitrary size but minimize the number of parts at every vertex.
An \emph{$ \ell $-local queue layout} (\emph{$ \ell $-local book embedding}) is one in which every vertex has incident edges in at most $ \ell $ queues (pages).
The \emph{local queue number} $ \lqn(G) $ (\emph{local page number} $ \lpn(G) $) is the smallest $ \ell $ for which there is an $ \ell $-local queue layout for $ G $.
Note that we have $ \lqn(G) \leq \qn(G) $ and $ \lpn(G) \leq \pn(G) $ as layouts of size $ \ell $ are also $ \ell $-local.

Finally, a \emph{$ k $-tree} is a $ (k + 1) $-clique or is obtained from a smaller $ k $-tree by choosing a clique $ C $ of size $ k $ and adding a new vertex $ u $ which is adjacent to all vertices of $ C $.
Fixing an arbitrary construction ordering, the vertex $ u $ is called a \emph{child} of $ C $, and $ C $ is called the \emph{parent clique} of $ u $.
We also say $ u $ is a child of each vertex of $ C $.
A child is called \emph{nesting} with respect to a vertex ordering if it is placed below its parent clique, and \emph{non-nesting} otherwise.
Note that $ k $-trees are exactly the maximal graphs with treewidth $ k $.
As local queue and page numbers are monotone, it suffices for us to investigate $ k $-trees, instead of arbitrary graphs of treewidth $k$.

\subsection{Related Work and Motivation}
\label{sec:relatedWork}

The notion of local ordered covering numbers unifies the concepts of local covering numbers and ordered covering numbers.
The first was introduced by Knauer and Ueckerdt~\cite{cover}, while existing research on the latter focuses on queue numbers and page numbers, which were established by Bernhart and Kainen~\cite{bookThickness} and Heath and Rosenberg~\cite{1queue}, respectively.

We first give a brief overview of global and local covering numbers as introduced in~\cite{cover}.
Consider a class of graphs $ \calG $, called \emph{guest class}, and an input graph $ H $.
We say the graph $ H $ is \emph{covered} by some covering graphs $ G_1, \ldots, G_t \in \calG $ if $ G_i $ is a subgraph of $ H $ for each $ i $ and every edge of $ H $ is contained in some covering graph, i.e.\ if $ G_1 \cup \cdots \cup G_t = H $.
The set of covering graphs is called an \emph{injective \calG-cover} of $ H $.
The \emph{global covering number} is the minimum number of covering graphs needed to cover a graph $ H $, that is the size of the smallest injective \calG-cover of $ H $.
For the \emph{local covering number}, we use \calG-covers of arbitrary size and minimize the number of covering graphs at every vertex.
For this, we say a \calG-cover for a graph $ H $ is \emph{$ \ell $-local} if every vertex is contained in at most $ \ell $ covering graphs.
Now, the \emph{local covering number} of a graph $ H $ with guest class \calG is defined as the smallest $ \ell $ such that there is an $ \ell $-local injective $ \calG $-cover of $ H $.

Many known graph parameters are covering numbers.
For instance, the thickness and outerthickness are global covering numbers for the guest classes of planar and outerplanar graphs, respectively~\cite{thicknessSurvey,outerthickness}.
In addition, all kinds of arboricity are global covering numbers for the guest class of the respective forests~\cite{arboricity,linearArboricity,caterpillarArboricity,starArboricity}.
The local covering number was considered for the guest classes of complete bipartite graphs~\cite{localBipartiteFishburn}, complete graphs~\cite{localClique}, and different forests~\cite{cover}.

We continue by summarizing known results on the queue number and local page number of planar graphs and graphs with bounded treewidth.
While every $ 1 $-queue graph is planar~\cite{1queue}, the maximum queue number among all outerplanar graphs is $ 2 $~\cite{compQS} and among all planar graphs it is between $ 4 $ and $ 49 $~\cite{planar3TreeQN,planarBoundedQN}.
The lower bound of $ 4 $ is obtained by a planar 3-tree.
Alam et~al.~\cite{planar3TreeQN} also show that every planar 3-tree admits a 5-queue layout.
Trees have queue number $ 1 $ using a BFS-ordering~\cite{1queue}, and BFS-orderings proved also useful for queue layouts of planar graphs~\cite{planarBoundedQN}, outerplanar graphs~\cite{compQS}, and graphs with bounded treewidth~\cite{treewidthExp}.
Rengarajan and Veni Madhavan~\cite{2trees} prove that every 2-tree admits a 3-queue layout, while Wiechert~\cite{treewidthExp} proves that this bound is tight.
More general, there is a graph with treewidth $ k $ and queue number at least $ k + 1 $ for each $ k > 1 $, while the best known upper bound is $ 2^k - 1 $~\cite{treewidthExp}.

The local version of page numbers was introduced and investigated in~\cite{localPageNumbers}.
The local page number of any graph is always near its maximum average degree, while the classical page number can be arbitrarily far off:
For any $ d \geq 3 $, there are $ n $-vertex graphs with local page number at most $ d + 2 $ but page number $ \Omega(\sqrt{d}n^{1/2 - 1/d}) $.
The maximum local page number for $k$-trees is at least $ k $ and at most $ k + 1 $, and for planar graphs it is either $ 3 $ or $ 4 $.

Our main motivation for defining local ordered covering numbers is to combine the well-studied notions of ordered covering numbers and local covering numbers and thereby continue research on both concepts.
The questions we ask for the new graph parameters naturally arise from those asked for the known concepts.
Studying ordered graphs and covering numbers is additionally motivated by applications in very-large-scale integration (VLSI) circuit design and bioinformatics~\cite{diogenes,bioOrdered,vlsiCover,bioCover}.
In addition, covers appear in network design~\cite{networkCover},
while queue layouts are closely related to 3-dimensional graph drawing~\cite{3dDrawing} and parallel multiplications of sparse matrices~\cite{matrixMult}.

\subsection{Contribution}
\label{sec:contribution}

We first observe that there are graphs whose local queue number is arbitrarily far from its queue number.
In addition, the local queue number is tied to the \emph{maximum average degree}, which is defined as 
$ \mad(G) = \max\{2 |E(H)| / |V(H)| \colon \allowbreak H \subseteq G, H \neq \emptyset\} $.
Both results are derived from the analogous results for local page numbers~\cite{localPageNumbers}, which is why we omit the proofs here. They can be found in \cref{sec:gap-mad}. 
\Cref{thm:mad} also implies that the local queue number is tied to the local page number, which is conjectured for the classical page number and classical queue number.

\begin{theorem}
    \label{thm:gap}
    For any $ d \geq 3 $ and infinitely many $ n $, there exist $ n $-vertex graphs with local queue number at most $ d + 2 $ but queue number $ \Omega(\sqrt{d}n^{1/2 - 1/d}) $.
\end{theorem}

\begin{theorem}
    \label{thm:mad}
    For any graph $ G $, we have
    \begin{equation*}
        \frac{\mad(G)}{4} \leq \lqn(G) \leq \frac{\mad(G)}{2} + 2.
    \end{equation*}
\end{theorem}

While the best upper bound for the queue number of $ k $-trees is $ 2^k - 1 $ due to Wiechert~\cite{treewidthExp}, \cref{thm:mad} already provides a linear upper bound for the local queue number of $k$-trees, which can be slightly improved.

\begin{theorem}
    \label{thm:lqn-ktree-upper}
    Every graph with treewidth $ k $ admits a $ (k + 1) $-local queue layout.
\end{theorem}

Suspecting that the bound in \cref{thm:lqn-ktree-upper} might be tight, we focus on lower bounds for the local queue number of $ k $-trees in \cref{sec:ktrees}.
Our main contribution is a tool that allows to focus on the construction of cliques with non-nesting children.
We use this to prove that \cref{thm:lqn-ktree-upper} is tight for $ k = 2 $ and that there are $ k $-trees whose local queue number is at least $ \lceil k/2 \rceil + 1 $ for $ k > 1 $.

\begin{theorem}
    \label{thm:qnl2tree}
    There is a graph with treewidth $ 2 $ and local queue number $ 3 $.
\end{theorem}

\begin{theorem}
    \label{thm:lqn-ktree-lower}
    For every $ k > 1 $, there is a graph $ G $ with treewidth $ k $ and local queue number at least $ \lceil k/2 \rceil + 1 $.
\end{theorem}

As the maximum average degree of planar graphs is strictly smaller than $ 6 $ and $ 2 $-trees are planar, \cref{thm:mad,thm:qnl2tree} bound the maximum local queue number of the class of planar graphs.
\begin{corollary}
    Every planar graph admits a $ 4 $-local queue layout and there is a planar graph whose local queue number is at least $ 3 $.
\end{corollary}

\section{The Local Queue Number of \texorpdfstring{$ k $}{k}-Trees}
\label{sec:ktrees}

We first provide a straight-forward construction for the upper bound of $ k + 1 $ for the local queue number of $ k $-trees, which proves \cref{thm:lqn-ktree-upper}.

\begin{proof}[of \cref{thm:lqn-ktree-upper}]
    We partition the edges of a $ k $-tree $ G $ into stars, each forming a queue.
    Consider an arbitrary construction ordering of $ G $ and let $ v_1, \dots, v_n $ denote the vertices of $ G $ in this ordering.
    For each vertex $ v_i \in V(G) $, $ i = 1, \dots, n $, we define a queue $ Q_i $ that contains all edges from $ v_i $ to its children,
    that is $ Q_i = \{v_i v_j \in E(G) \colon i < j\} $.
    Choosing an arbitrary vertex ordering yields a queue layout since edges of a star cannot nest.
    The layout is $ (k + 1) $-local since every vertex has at most $ k $ neighbors with smaller index.
\end{proof}

As our main tool for constructing $ k $-trees with large local queue number, we introduce a sequence of two-player games that are adaptions of a game introduced by Wiechert~\cite{treewidthExp}.
Taking turns, Alice constructs a $ k $-tree, which is laid out by Bob.
The rules for Alice stay the same in all games, whereas the rules for Bob include only the first $ n $ conditions in the $ n $-th game (see below).
We always assume that Bob has an optimal strategy and prepare Alice to react on all possible moves of Bob.
That is, when we write \emph{Alice wins}, then we mean that she wins regardless of the layout Bob chooses.

\begin{figure}
    \centering
    \includegraphics{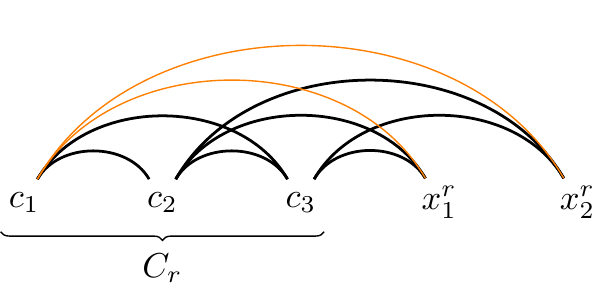}
    \caption{
        Notation for the Games~\labelcref{cond:layout,cond:firstRound,cond:consecutive,cond:twinColors,cond:differentColors}. 
        $ C_r $ is the parent clique that Alice chooses in the $ r $-th round and she chooses to add $ m_r = 2 $ children.
        The two children $ x_1^r $ and $ x_2^r $ are twin vertices. 
        The two orange (thin) edges are twin edges.
    }
    \label{fig:notation}
\end{figure}

The graph which is laid out in the $ r $-th round of any game is denoted by $ G_r $, the layout Bob creates by $ (\prec_r, \calQ_r) $. 
In the beginning, there is an initial clique $ \cinit $ whose edges are assigned to arbitrary queues.
In particular, we have $ G_0 = \cinit $.
The notation we introduce for the games is summarized in \cref{fig:notation}.
In the $ r $-th round, Alice chooses a $ k $-clique $ C_r $ from the current graph $ G_{r - 1} $ and an integer $ m_r $.
Now, $ m_r $ new vertices $ x_1^r, \dots, x_{m_r}^r $ are introduced and become adjacent to the vertices of $ C_r $.
The clique $ C_r $ is the \emph{parent clique} of the new vertices and edges.
Vertices with the same parent clique are called \emph{twin vertices} and two edges that share a vertex in the parent clique and are introduced in the same round are called \emph{twin edges}.
Then, in the $n$-th game, Bob inserts the new vertices into the current vertex ordering and assigns the new edges to queues satisfying the first $n$ of the following conditions:
\begin{enumerate}[label=(\roman*)]
    \item \label[cond]{cond:layout} 
        The layout $ (\prec_r, \calQ_r) $ is an $ \ell $-local queue layout of $ G_r $.
    \item \label[cond]{cond:firstRound}
        In the first round, all new vertices are placed to the right of $ \cinit $.
        Without loss of generality, we have $ \cinit \prec x_1^1 \prec \dots \prec x_{m_1}^1 $.
    \item \label[cond]{cond:consecutive}
        The new vertices $ x_1^r, \dots, x_{m_r}^r $ are inserted consecutively, i.e.\ $ y $ is not in the span of $ x_1^r, \dots, x_{m_r}^r $ for all vertices $ y \in V(G_{r - 1}) $ from the previous rounds.
    \item \label[cond]{cond:twinColors}
        Each two twin edges are assigned to the same queue.
    \item \label[cond]{cond:differentColors}\label[cond]{cond:right}
        The new vertices are placed to the right of their parent clique. 
        Without loss of generality, we have $ C_r \prec x_1^r \prec \dots \prec x_{m_r}^r $.
        The edges between a vertex $ x_i^r $, $ i \in \{1, \dots, m_r\} $, and its parent clique $ C_r $ are assigned to pairwise different queues.
        In particular, if $\ell = k$, then Bob cannot introduce new queues at vertex $ x_i^r $ in the following rounds.
\end{enumerate}

Alice wins the $ n $-th game if Bob cannot extend the layout without violating one of the first $ n $ conditions.
In particular, if Alice wins the first game, this implies the existence of a $ k $-tree with local queue number $ \ell + 1 $.
However, the first game is the hardest for Alice, whereas the games become easier when Bob's moves are more restricted.
During the proofs, we decide what Alice does but cannot control Bob's moves.
We say that Bob \emph{has to} act in a certain way if Alice wins otherwise.

We now set out to show how Alice wins \game{cond:layout} for $ k = \ell = 2 $.
We first present a $2$-tree with which Alice wins \game{cond:differentColors} and then show how to augment it until arriving at a $2$-tree with which she wins the first game.

\begin{lemma}
    There is a graph with which Alice wins \game{cond:differentColors} for $ k = \ell = 2 $.
    \label{lem:alice-wins-2}
\end{lemma}

\begin{proof}
    \begin{figure}
        \centering
        \includegraphics{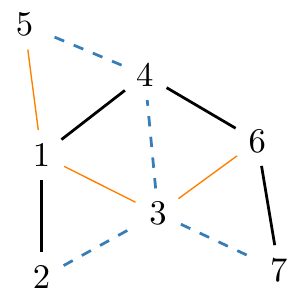}\hfill
        \includegraphics{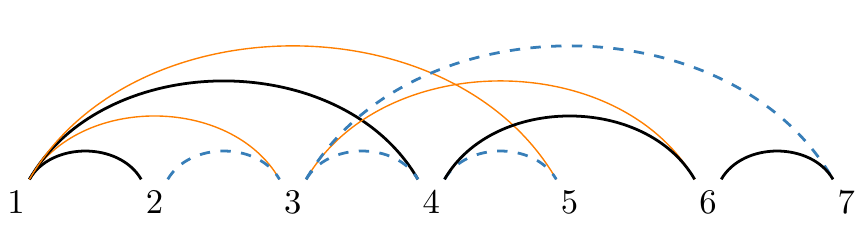}
        \caption{2-tree for \game{cond:differentColors} (left) and the layout chosen by Bob (right)}
        \label{fig:AliceWinsGraph}
        \label{fig:AliceWinsLayout}
    \end{figure}

    Consider the $ 2 $-tree presented in \cref{fig:AliceWinsGraph}.
    The edge $ \{1, 2\} $ is the initial clique $\cinit$ and Alice introduces in five rounds the vertices one-by-one in the order indicated by their number.
    We first argue why we may assume that Bob does not assign any new edge to the same queue as its parent clique.
    Assume that in some round $r$, Bob chooses to assign an edge $e$ between the parent clique $ C_r $ and its child $x_1^r$ to the same queue $ Q $ as $C_r$.
    \Cref{cond:differentColors} ensures that the other edge $e'$ between $C_r$ and $x_1^r$ is assigned to a different queue $ Q' $.
    Thus, both endpoints of $ e' $ have incident edges in the same two queues $ Q $ and $ Q' $.
    In particular, Bob cannot introduce new queues for edges adjacent to $ e' $.
    \Cref{cond:differentColors} further states that Bob cannot introduce new queues at any vertex that does not belong to the initial clique.
    It follows that Bob can use only the two queues $ Q $ and $ Q' $ for any subgraph that starts with $ e' $ as initial clique.
    Now, Alice wins by constructing the $ 2 $-tree with queue number $ 3 $ provided by Wiechert~\cite{treewidthExp} using $e'$ as the initial clique.
    
    Hence, Bob has no choice when assigning the edges to queues.
    The vertex ordering is determined by the rules except for the placement of vertex $ 6 $ which may be placed between the vertices $ 4 $ and $ 5 $.
    In this case, however, the edges $ \{1, 5\} $ and $ \{3, 6\} $ form a rainbow.
    Thus, Bob chooses the layout shown in \cref{fig:AliceWinsLayout}, which again has a rainbow (edges $ \{3, 7\} $ and $ \{4, 5\} $) and Alice wins \game{cond:differentColors}.
\end{proof}

We give the reductions from \game{cond:differentColors} to \game{cond:firstRound} in a more general way, so we can reuse them for subsequent lemmas.

\begin{lemma}
    Let $ k > 1 $ and $ \ell \leq k $.
    If Alice wins \game{cond:differentColors}, then she also wins \game{cond:firstRound}.
    \label{lem:games}
\end{lemma}

\begin{proof}
    We prove that Alice wins the $ n $-th game provided a strategy with which she wins the $ (n + 1) $-st game.
    We do this by adapting the strategy such that the $ (n + 1) $-st condition is satisfied in the $ n $-th game.
    Alternatively, we prove that Alice wins the $ n $-th game if Bob does not act as claimed in the $ (n + 1) $-st condition.
    Thus, we may assume that the $ (n + 1) $-st condition holds and apply the given strategy for the $ (n + 1) $-st game.
    
    \gameRed{cond:differentColors}{cond:twinColors}
    We assume that Alice has a strategy to win \game{cond:differentColors} and present how Alice adapts her moves to ensure \cref{cond:differentColors} in \game{cond:twinColors}.
    We exploit the vertex placement and queue assignment of twin vertices and twin edges to prove that Bob creates a rainbow unless he places new vertices to the right of their parent clique.
    We thereby observe that each two non-twin edges introduced in the same round are assigned to different queues.
    We proceed by induction on the number of rounds $ r $.
    The vertex placement in the first round is established by \cref{cond:firstRound}.
    In each of the succeeding rounds, let Alice increase the number $ m_r $ of added vertices by 2.
    
    Consider a set of twin vertices $ X = \{x_0^p, x_1^p, \dots, x_{m_p + 1}^p\} $ that were added to a clique $ C_p $ in a former round $ p $.
    To simplify the notation, we write $ x_i $ instead of $ x_i^p $ and $ m $ instead of $ m_p $.
    The vertices of $ C_p $ are denoted by $ c_1 \prec \dots \prec c_k $.
    By induction, we have $ C_p \prec X $.
    Observe that this already implies that the edges at each $ x_i $ are assigned to pairwise different queues since $ c_h x_1 $ and $ c_j x_0 $ nest for $ 1 \leq h < j \leq k $ and twin edges are assigned to the same queue due to \cref{cond:twinColors} (see \cref{fig:differentColors}).
    In particular, Bob has to choose one of the existing queues for new edges incident to $ x_i $.
    
    \begin{figure}
        \centering
        \includegraphics{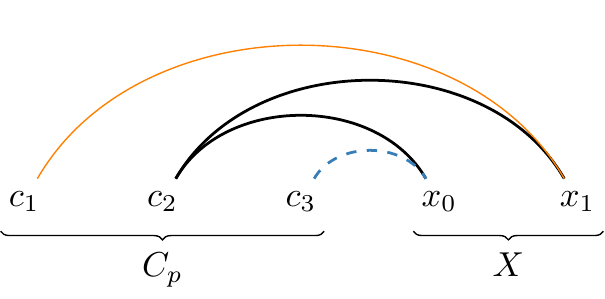}
        \caption{
            The twin edges $ c_2 x_0 $ and $ c_2 x_1 $ are in the same queue. 
            All edges between $ C_p $ and $ X $ have a twin edge forming a rainbow with one of them.
        }
        \label{fig:differentColors}
    \end{figure}
    
    \begin{figure}
        \centering
        \includegraphics{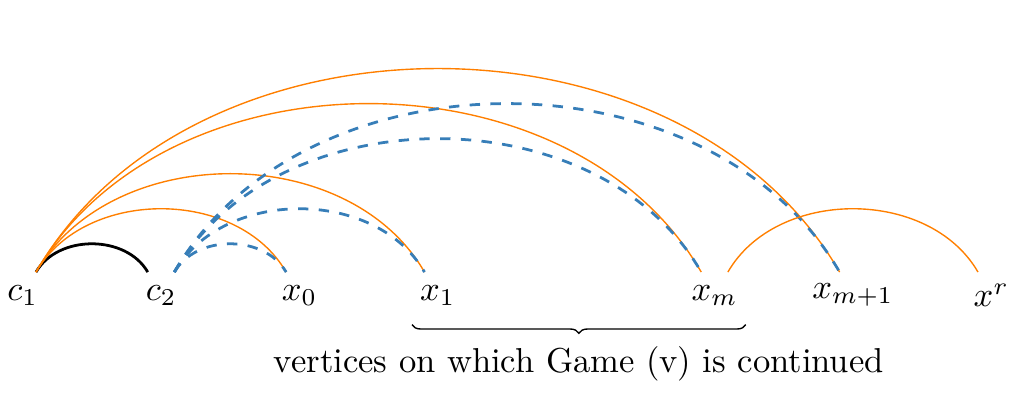}
        \caption{
            Vertices $ x_0, \dots, x_{m + 1} $ and their parent clique $ C_p $. 
            A child $ x^r $ of $ x_i $, $ i \in \{1, \dots, m\} $, can only be placed to the right of $ x_{m + 1} $.
        }
        \label{fig:gameRight}
    \end{figure}

    Alice now continues applying her strategy for \game{cond:right} using the vertices $ x_1, \dots, x_m $.
    Assume that, in some round $ r $, she chooses a clique $ C_r $ that contains $ x_i $ for some $ i \in \{1, \dots, m\} $ (see \cref{fig:gameRight}).
    Bob has to insert a child $ x^r $ of $ C_r $ into the vertex ordering and has to choose a queue $ Q $ for $ x_i x^r $ which already contains an edge $ c x_i $ with $ c \in V(C_p) $.
    If Bob places $ x^r $ between $ c $ and $ x_{m + 1} $, then $ c x_{m + 1} $ and $ x_i x^r $ form a rainbow in the same queue.
    If Bob places $ x^r $ to the left of $ c $, then $ c x_0 $ and $ x_i x^r $ nest and are assigned to the same queue.
    Hence, he has to place the new vertices to the right of $ x_{m + 1} $ and therefore satisfies \cref{cond:right}.
    Thus, Alice may apply her strategy for \game{cond:right} to win.
    
    \gameRed{cond:twinColors}{cond:consecutive}
    We consider \game{cond:consecutive} and show how Alice ensures \cref{cond:twinColors}, i.e.\ that each two twin edges are assigned to the same queue.
    For this, recall that each vertex has incident edges in at most $ \ell $ different queues.
    Thus, Bob has at most $ \ell^k $ possibilities to assign the $ k $ edges of a new vertex to queues.
    Alice multiplies the number $ m_r $ of added vertices by $ \ell^k $ in each round and thus finds $ m_r $ twin vertices whose twin edges are assigned to the same queues.
    She continues the game on those with the same strategy with which she wins \game{cond:twinColors} and ignores the others.
    
    \gameRed{cond:consecutive}{cond:firstRound}
    \Cref{cond:consecutive} forces Bob to insert twin vertices consecutively into the vertex ordering.
    Let $ m_r $ be the number of vertices that Alice adds to the current graph in the $ r $-th round of \game{cond:firstRound}.
    To simulate \cref{cond:consecutive} in \game{cond:firstRound}, Alice chooses to add $ (m_r + 1) \cdot |V(G_{r - 1})| $ vertices instead of only $ m_r $.
    By pigeonhole principle, Bob places at least $ m_r $ new vertices consecutively. 
    Alice now uses her strategy from \game{cond:consecutive} to win \game{cond:firstRound}.
\end{proof}

Finally, the following lemma justifies to introduce \cref{cond:firstRound} if $ k = \ell = 2 $.
That is, we use \cref{lem:startEdge} to win \game{cond:layout} provided a strategy to win \game{cond:firstRound}.
We choose $ s = 2 m_1 $, where $ m_1 $ is the number of vertices that are introduced in the first round of \game{cond:firstRound}.
The clique given by \cref{lem:startEdge} serves as initial clique \cinit.
Without loss of generality, \cinit has $ m_1 $ children to the right and therefore satisfies \cref{cond:firstRound}.

\begin{lemma}
    \label{lem:startEdge}
    For any $ s > 0 $, there is a $ 2 $-tree $ G $ such that for every $ 2 $-local queue layout there is an edge with at least $ s $ non-nesting children.
\end{lemma}

\begin{proof}
    An \emph{$ m $-ary $ 2 $-tree of depth $ t $} is constructed as follows. 
    We start with an edge whose depth, and also the depth of its endpoints, is defined to be 0.
    For $ 0 < i \leq t $, depth-$ i $ edges are introduced inductively by adding $ m $ children to each depth-$ (i - 1) $ edge.
    The depth of the new children is $ i $.
    $ G $ is an $ (s + 4) $-ary 2-tree of depth 6 and is partly shown in \cref{fig:belowGraph}.
    
    For the sake of contradiction, consider a 2-local queue layout of $ G $ such that every edge of depth less than 6 has at least five nesting children.
    We shall find a rainbow in one of the queues.
    
    \begin{figure}
        \centering
        \includegraphics{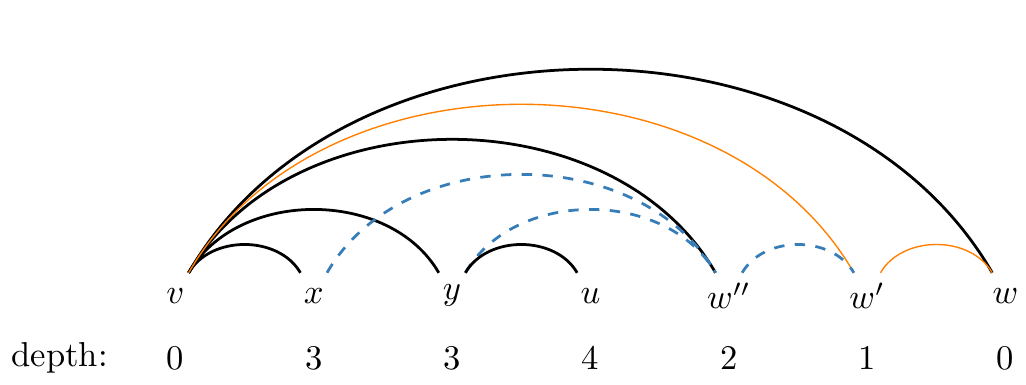}
        \caption{%
            2-tree with nesting children. 
            The edges in \nameColorA (thick), \nameColorB (thin), and \nameColorC (dashed) correspond to the respective queues.
            The edges $ vx $ and $ vy $ could also be \nameColorB, $ yu $ is \nameColorA, \nameColorB, or \nameColorC and creates a rainbow in either case.
        }
    \label{fig:belowGraph}
    \end{figure}

    Let $ vw $ denote the depth-0 edge and let $ w' $ be a nesting child of $ vw $.
    The initial edge $ vw $ is assigned to some queue \qa.
    We now assume that $ vw' $ is assigned to a different queue \qb and handle the other case later.
    Consider a nesting child $ w'' $ of $ vw' $.
    Next, we have five depth-3 children of $ vw'' $ which are placed below their parent edge by assumption.
    Since the layout is 2-local, there are only four possible combinations how the edges of depth 3 incident to $ v $ and $ w'' $ can be assigned to queues.
    By pigeonhole principle, there are two vertices $ x $ and $ y $ with $ vx $ and $ vy $ assigned to the same queue $ Q $, where $ Q = \qa $ or $ Q = \qb $, and $ xw'', yw'' \in \qc $ for some queue $ \qc $.
    Note that $ \qc \neq \qa, \qb $ as otherwise this would create a rainbow in the respective queue.
    Without loss of generality, we have $ x \prec y $.
    
    Finally, consider a nesting child $ u $ of the edge $ yw'' $.
    Since the layout is 2-local, we have $ yu \in Q $ or $ yu \in \qc $, that is $ yu $ is assigned to \qa, \qb, or \qc.
    In all three cases, there is a rainbow in the respective queue.
    
    To end the proof, consider the case that $ vw' $ is assigned to \qa.
    If $ vx, vy \in \qa $, then the argumentation above still works.
    Otherwise, use $ vx $ instead of $ vw' $.
    
    We conclude that in every 2-local queue layout, one of the edges has fewer nesting children than we used in the construction, in particular fewer than five.
    Therefore, all further children are non-nesting.
\end{proof}

\Cref{lem:alice-wins-2,lem:games,lem:startEdge} together show that Alice wins \game{cond:layout} for $ k = \ell = 2 $.
That is, there is a $ 2 $-tree with local queue number $ 3 $, which proves \cref{thm:qnl2tree}.

Since Bob may use more queues for larger $ k $, the approach for \cref{lem:alice-wins-2} using a $ k $-tree with queue number $ k + 1 $ works only for $ k = 2 $.
However, we introduce two new conditions that restrict Bob's moves further and offer Alice a way to win \game{cond:differentColors} for any $ k > 1 $ and $ \ell \leq k $.

For Games~\labelcref{cond:sisters,cond:diverseChildren}, we change the initial setup.
Instead of starting with only one clique, we start with two cliques and proceed on both cliques in parallel.
We thereby get two copies of the same graph (a \emph{left graph} and a \emph{right graph}), where each vertex, edge, and clique has a corresponding copy in the other graph.
In the $ r $-th round, Alice now chooses a $ k $-clique in the left graph and its copy in the right graph, and adds $ m_r $ new vertices to each.
\Cref{cond:layout,cond:firstRound,cond:consecutive,cond:twinColors,cond:differentColors} apply to the left graph and to the right graph independently.
In addition, Bob has to satisfy the following conditions (only the first for \game{cond:sisters} and both for \game{cond:diverseChildren}):

\begin{figure}
    \centering
    \includegraphics{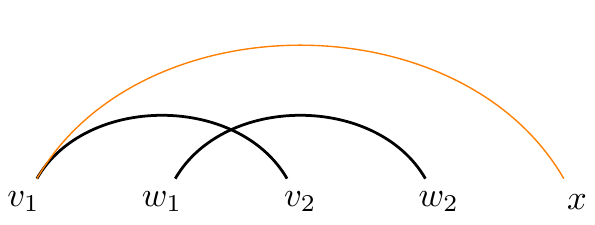}
    \caption{Two sister edges and a child $ x $ to the right of both}
    \label{fig:diverseChildren}
\end{figure}
    
\begin{enumerate}[label=(\roman*), resume]
    \item \label[cond]{cond:sisters}
        Each two cliques $ C $ and $ C' $ that are copies of each other are laid out alternatingly, that is $ c_1 \prec c_1' \prec c_2 \prec c_2' \prec \dots \prec c_k \prec c_k' $ for vertex sets $ V(C) = \{c_1, \dots, c_k\} $ and $ V(C') = \{c_1', \dots, c_k'\} $.
        In particular, new vertices and their copies are placed in the order of their parent cliques to the right of both parent cliques.
        Every edge is assigned to the same queue as its copy.
    \item \label[cond]{cond:diverseChildren}
        Consider an edge $ v_1 v_2 $ with its copy $ w_1 w_2 $ (see \cref{fig:diverseChildren}).
        If there is a child $ x $ of $ v_1 $ to the right of both edges, then the edges $ v_1 x $ and $ v_1 v_2 $ are assigned to different queues.
\end{enumerate}

We show how Alice wins Games~\labelcref{cond:differentColors,cond:sisters,cond:differentColors} for any $ k > 1 $ and $ \ell \leq k $ in \cref{sec:games}. 
The two new games together with \cref{lem:games} lead to the following lemma showing that it suffices to find a $ k $-tree with non-nesting children to prove lower bounds.
We remark that placing children outside their parent clique would be a natural way to avoid nesting edges.
Note that \cref{cond:firstRound} is justified by the requirement of \cref{lem:require-nonnesting}.

\begin{lemma}
    \label{lem:require-nonnesting}
    Let $ k > 1 $ and $ \ell \leq k $.
    Assume that for any $ s > 0 $ there is a $ k $-tree such that for every $ \ell $-local queue layout, there is a $ k $-clique with at least $ s $ non-nesting children.
    Then, there is a $ k $-tree with local queue number at least $ \ell + 1 $.
\end{lemma}

Observe that we can embed $ k' $-trees in $ k $-trees with $ k' < k $.
A queue layout of the $ k $-tree, however, can result in additional restrictions to the queue layout of the embedded $ k' $-tree.
In particular, if every $ \ell $-local queue layout of some $ k $-tree contains a $ k' $-tree having a $ k' $-clique with non-nesting children, we may apply \cref{lem:require-nonnesting} to this clique.
The resulting $ k' $-tree with local queue number at least $ \ell + 1 $ can then be augmented to a $ k $-tree with local queue number at least $ \ell + 1 $.
This leads to the following strengthening of \cref{lem:require-nonnesting}.

\begin{lemma}
    \label{lem:require-nonnesting-subgraph}
    Let $ 1 < k' \leq k $ and $ \ell \leq k' $.
    Assume that for any $ s > 0 $ there is a $ k $-tree $ G $ such that for every $ \ell $-local queue layout, $ G $ contains a $ k' $-clique with at least $ s $ non-nesting children.
    Then, there is a $ k $-tree with local queue number at least $ \ell + 1 $.
\end{lemma}

Note that adding $ 2s $ children to a $ k $-clique yields a $ k $-tree that contains a $ \lceil k/2 \rceil $-clique with at least $ s $ non-nesting children.
Thus, \cref{lem:require-nonnesting-subgraph} with $ \ell = k' = \lceil k/2 \rceil $ proves \cref{thm:lqn-ktree-lower}.

\section{Conclusions}
\label{sec:conclusions}

Based on the notions of queue numbers and local covering numbers, we introduced the \emph{local queue number} as a novel graph parameter.
We presented a tool to deal with $ k $-trees which led to the construction of a $ 2 $-tree with local queue number $ 3 $.
This strengthens the lower bound of $ 3 $ for the queue number of $ 2 $-trees due to Wiechert~\cite{treewidthExp}.
It remains open whether there are $ k $-trees with local queue number $ k + 1 $ for $ k > 2 $.
Given \cref{lem:require-nonnesting}, this is equivalent to the existence of $ k $-trees that have a $ k $-clique with non-nesting children for any $ k $-local queue layout.
That is, if every $ k $-tree admits a $ k $-local queue layout, then every $ k $-tree also admits a $ k $-local queue layout such that all children are placed below their parent clique.
As such a vertex placement produces many nesting edges, this does not seem to be a promising strategy.

\begin{question}
    What is the maximum local queue number of treewidth-$k$ graphs?
\end{question}

There is a third parameter that is closely related to queue numbers and local queue numbers.
For the \emph{union queue number} $ \uqn(G) $ of a graph $ G $, we define a \emph{union queue} to be a vertex-disjoint union of queues and then minimize the number of union queues that are necessary to cover all edges of $ G $.
As we minimize the size of the cover, one could consider the union queue number close to the queue number.
Surprisingly, the union queue number is tied to the local queue number and we have a linear upper bound on the union queue number of $ k $-trees.
We refer to~\cite{localPageNumbers} for an analogous proof for local and union page numbers.
This observation has interesting consequences for queue layouts of $ k $-trees.
If the best known upper bound of $ 2^k - 1 $~\cite{treewidthExp} is tight, then there are queue layouts consisting of linearly many union queues but at least exponentially many queues.

On the other hand, \cref{lem:require-nonnesting} might be extendable for $ \ell $-queue layouts with $ \ell > k $.
\Cref{cond:differentColors} is crucial for this as it forbids Bob to introduce new queues.
If Bob may use more than $ k $ queues, however, the given proof fails.
Note that the requirement of a $ k $-clique with non-nesting children is satisfied by a construction presented by Wiechert~\cite[Lemma 10]{treewidthExp}.

Finally, the presented 2-tree also serves as a witness that there are planar graphs with local queue number at least 3.
However, it is open whether this can be improved to 4.

\begin{question}
    Is there a planar graph with local queue number $ 4 $?
\end{question}

\bibliography{references}

\appendix
\section*{Appendix}

\section{Proofs of \cref{thm:mad,thm:gap}}
\label{sec:gap-mad}

The proofs in this section follow the proofs of the analogous theorems on the local page number in \cite{localPageNumbers}.
We start with a proof of \cref{thm:mad} which claims that the local queue number is tied to the maximum average degree.
For bounding the local queue number from below, consider a nonempty subgraph $ H $ of a graph $ G $ of local queue number
$ \lqn(G) = k $. 
Let \calQ denote the set of queues of a $ k $-local queue layout.
For a queue $ Q \in \calQ $, let $ V_Q $ denote the set of vertices that have incident edges in $ Q $.
Heath and Rosenberg~\cite{1queue} show that $ Q $ contains at most $ 2 |V_Q| - 3 $ edges.
For the number of edges of $ H $, we thus have
\begin{equation*}
    \label{eq:edges}
    |E(H)|
    \leq \sum\limits_{Q \in \calQ} (2 \, |V_Q| - 3)
    \leq 2 k |V(H)| - 3 |\calQ|
    \leq 2 \lqn(G) \cdot |V(H)|.
\end{equation*}%
It follows that the local queue number of $ G $ is lower-bounded by $ |E(H)| / (2 |V(H)|) $, i.e.\ by a forth of the average degree, for any nonempty subgraph $ H $.
Hence, we have $ \lqn(G) \geq \mad(G)/4 $.

Perhaps surprisingly, there is also an \emph{upper} bound on the local queue number in terms of the graph's density.
Nash-Williams~\cite{arboricity} proves that any graph $G$ edge-partitions into $k$ forests if and only if
\[
 k \geq \max\left\{ \frac{|E(H)|}{|V(H)|-1} \mid H \subseteq G, |V(H)| \geq 2 \right\}.
\]
The smallest such $k$, the \emph{arboricity} $\arb(G)$ of $G$, thus satisfies $\frac{1}{2}\mad(G) < \arb(G) \leq \frac{1}{2}\mad(G)+1$.
The \emph{star arboricity} $\sa(G)$ of $G$ is the minimum $k$ such that $G$ edge-partitions into $k$ star forests.
Using the covering number framework, Knauer and Ueckerdt~\cite{cover} introduce the corresponding local covering number, the \emph{local star arboricity} $\sa_\ell(G)$, as the minimum $k$ such that $G$ edge-partitions into some number of stars, but with each vertex having an incident edge in at most $k$ of these stars.
It is known that the local star arboricity can be bound in terms of the arboricity as $ \arb(G) \leq \sa_\ell(G) \leq \arb(G)+1 $~\cite{cover}.
To find a suitable queue layout, take an arbitrary spine ordering and an edge-partition into stars.
Observe that no two edges of a star nest, regardless of the vertex ordering.
Together, we obtain $\lqn(G) \leq \sa_\ell(G) \leq \arb(G) + 1 \leq \frac{1}{2}\mad(G) + 2 $.
This concludes the proof of \cref{thm:mad}.

Heath et al.~\cite{compQS} show that for every $ d \geq 3 $ and infinitely many $ n $, there exist $ d $-regular $ n $-vertex graphs with queue number at least $ \Omega(\sqrt{d}n^{1/2 - 1/d}) $.
On the other hand, the maximum average degree of any $ d $-regular graph $ G $ is $ d $.
With \cref{thm:mad}, it follows that $ \lqn(G) \leq d/2 + 2 $, which proofs \cref{thm:gap}.

\section{Games~\labelcref{cond:sisters,cond:diverseChildren}}
\label{sec:games}

\begin{figure}
    \centering
    \includegraphics{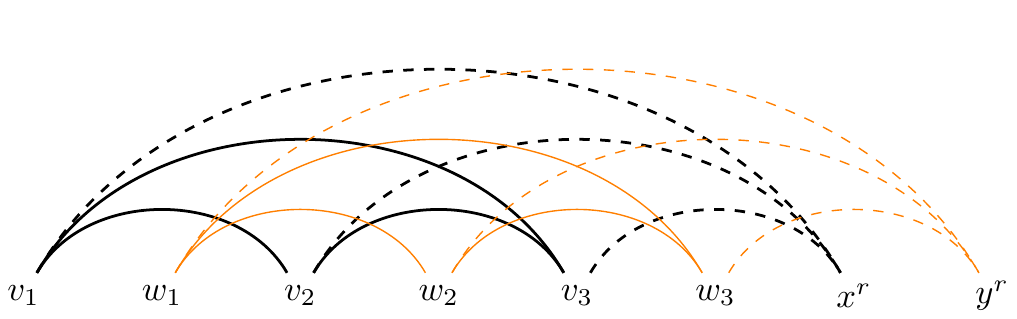}
    \caption{
        Two copies of a clique (black solid for the left graph and orange solid for the right graph) with a child each.
        The child $ y^r $ is a copy of $ x^r $, the vertex $ w_i $ is a copy of $ v_i $, and the edge $ w_i y^r $ is a copy of $ v_i x^r $ for $ i = 1, 2, 3 $.
        In addition, the $ 3 $-cliques containing $ y^r $ are copies of the respective $ 3 $-cliques containing $ x^r $.
    }
    \label{fig:notationCopy}
\end{figure}

We show how Alice wins Games~\labelcref{cond:differentColors,cond:sisters,cond:diverseChildren} for any $ k > 1 $ and $ \ell \leq k $.
Recall that we change the initial setup for the last two games.
See \cref{fig:notationCopy} for an illustration of the upcoming notation.
Instead of starting with only one clique, we start with two cliques $ C_v $ and $ C_w $ with vertices $ v_1 \prec \dots \prec v_k $ and $ w_1 \prec \dots \prec w_k $, respectively.
Without loss of generality, we have $ v_1 \prec w_1 $ and we refer to $ C_v $ as the \emph{left initial clique} and to $ C_w $ as the \emph{right initial clique}.
We proceed on both cliques in parallel and thereby get two copies of the same graph (a \emph{left graph} and a \emph{right graph}).
Consider a vertex $ v $ in the left graph and a vertex $ w $ in the right graph.
We say that $ v $ and $ w $ are \emph{copies} of each other if they are vertices in the respective initial clique having the same position, i.e.\ $ v = v_i $ and $ w = w_i $ for some $ i \in \{1, \dots, k\} $, or if they are introduced in the same round and have the same position with respect to their twin vertices.
That is, if vertices $ x_1^r \prec \dots \prec x_{m_r}^r $ are added to the left graph and vertices $ y_1^r \prec \dots \prec y_{m_r}^r $ are added to the right graph, then $ x_i^r $ and $ y_i^r $ are copies for each $ i \in \{1, \dots, m_r\} $.
Two edges whose endpoints are copies and two $ k $-cliques consisting of copied vertices are also said to be \emph{copies} of each other.
In particular, the two initial cliques are copies.

In the $ r $-th round, Alice now chooses a $ k $-clique in the left graph and its copy in the right graph.
As before, she chooses an integer $ m_r $.
To each of the chosen cliques, we add $ m_r $ new vertices.

\subsubsection{\game{cond:diverseChildren}.}
    \begin{figure}
        \centering
        \includegraphics{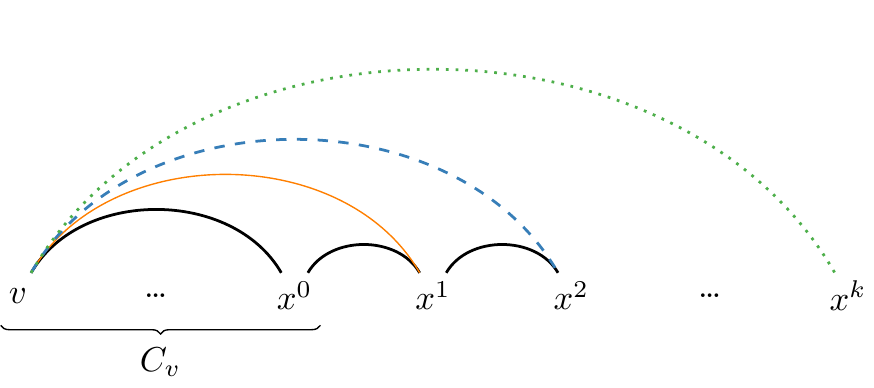}
        \caption{Left graph that Alice constructs to win \game{cond:diverseChildren}. The leftmost vertex $ v $ has incident edges in $ k + 1 $ queues.}
        \label{fig:aliceWins}
    \end{figure}

    We show how Alice wins \game{cond:diverseChildren}.
    Parts of the constructed graph are shown in \cref{fig:aliceWins}.
    Let $ v $ denote the leftmost vertex of the left initial clique $ C_v $.
    Alice adds vertices $ x^1, \dots, x^\ell $ to the left graph in $ \ell $ rounds, and the respective copies to the right graph.
    The rightmost vertex of $ C_v $ is denoted by $ x^0 $.
    In the $ i $-th round, Alice chooses a clique that contains $ v $ and the vertex $ x^{i - 1} $.
    By \cref{cond:sisters}, Bob inserts $ x^i $ to the right of $ x^{i - 1} $ and its copy.
    Now, \cref{cond:diverseChildren} ensures that Bob has to choose pairwise different queues for the edges $ v x^i $, $ i \in \{0, \dots, \ell\} $.
    Hence, $ v $ has incident edges in $ \ell + 1 $ queues and Alice wins.
    
    \gameRed{cond:diverseChildren}{cond:sisters}
    We prove that Alice wins \game{cond:sisters} if Bob does not act as claimed in \cref{cond:diverseChildren}.
    For this, consider an edge $ v_1 v_2 $ and its copy $ w_1 w_2 $, both assigned to some queue $ Q $.
    By \cref{cond:sisters}, we have $ v_1 \prec w_1 \prec v_2 \prec w_2 $ (see \cref{fig:diverseChildren}).
    Now, consider a child $ x $ of $ v_1 $ that is placed to the right of $ w_2 $.
    Since $ v_1 x $ and $ w_1 w_2 $ nest, Bob cannot assign $ v_1 x $ to $ Q $.
    
    \gameRed{cond:sisters}{cond:differentColors}
    Given a strategy to win \game{cond:sisters}, we play \game{cond:differentColors} and show how to apply the same strategy there.
    Note that Alice wins \game{cond:differentColors} if $ \ell < k $ since Bob needs to assign the $ k $ edges between a child and its parent clique to pairwise different queues but may use only $ \ell $ queues at every vertex.
    We thus may assume $ \ell = k $ for \game{cond:differentColors}.
    Recall that the difference between \game{cond:differentColors,cond:sisters} is not only the additional \cref{cond:sisters} but also that we use two initial cliques instead of only one.
    First, we argue why we may start with two cliques and why we may assume that their vertices are laid out alternatingly.
    We start with the initial clique $ \cinit $ of \game{cond:differentColors} whose vertices we denote by $ c_1, \dots, c_k $ and construct two cliques that serve as initial cliques for \game{cond:sisters}.
    For this, Alice creates many cliques that are pairwise vertex-disjoint.
    A clique with vertices $ v_1, \dots, v_k $ is created by choosing $ c_i, \dots, c_k, v_1, \dots, v_{i - 1} $ as parent clique for creating vertex $ v_i $.
    Due to \cref{cond:differentColors}, every queue contains a vertex in \cinit, and thus Bob can use at most $ k^2 $ queues.
    Note that for each clique $ k^2 $ edges are created (including the edges to \cinit) which Bob needs to assign to queues.
    We aim to find two cliques $ C_v $ and $ C_w $ with vertices $ v_1 \prec \dots \prec v_k $, respectively $ w_1 \prec \dots \prec w_k $, such that the edges $ v_i v_j $ and $ w_i w_j $ are assigned to the same queue for $ 1 \leq i < j \leq k $.
    In addition, we want the edges $ cv_i, cw_i $ to \cinit to be assigned to the same queue for each $ c \in V(\cinit) $ and $ i \in \{1, \dots, k\} $.
    By pigeonhole principle, creating more than $ (k^2)^{k^2} $ disjoint cliques yields two such cliques $ C_v $ and $ C_w $.
    
    \begin{figure}
        \centering
        \includegraphics{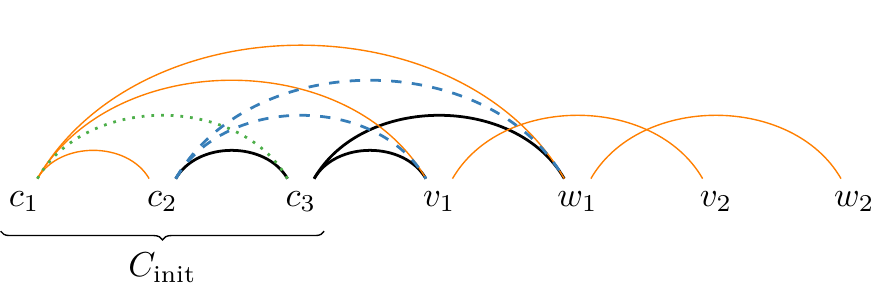}
        \caption{
            The vertices $ v_1 $ and $ w_1 $ have edges to \cinit in $ k $ different queues. 
            Bob has to reuse one of the queues for $ v_1 v_2 $ and $ w_1 w_2 $ and thus has to place $ v_2 $ and $ w_2 $ to the right.
        }
        \label{fig:twoInitialCliques}
    \end{figure}

    Now, we consider the vertex ordering of $ C_v $ and $ C_w $ (see \cref{fig:twoInitialCliques}).
    Without loss of generality, we have $ v_1 \prec w_1 $.
    Between $ w_1 $ and its parent clique, we have $ k $ pairwise different queues (\cref{cond:differentColors}).
    Thus, Bob has to assign $ w_1 w_2 $ (and also $ v_1 v_2 $) to one of these queues to keep the layout $ k $-local.
    Bob now has to place $ v_2 $ to the right of $ w_1 $ since otherwise the edge $ v_1 v_2 $ nests below all edges between $ w_1 $ and its parent clique.
    Since $ v_1 v_2 $ and $ w_1 w_2 $ are in the same queue, Bob has to place $ w_2 $ to the right of $ v_2 $.
    Inductively, we get $ v_1 \prec w_1 \prec v_2 \prec w_2 \prec \dots \prec v_k \prec w_k $.
    
    \begin{figure}
        \centering
        \includegraphics{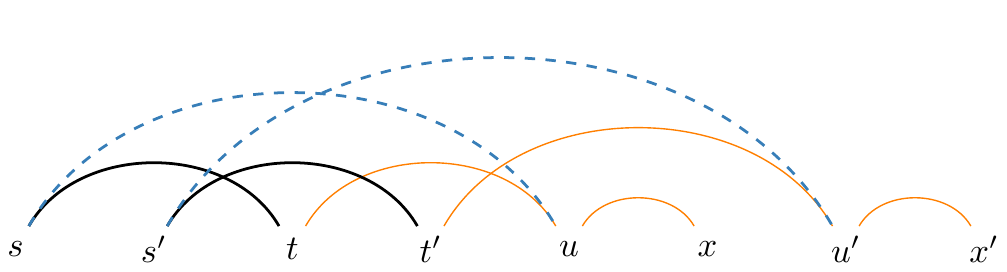}
        \caption{
            Two sister cliques $ C $ ($ s, t, u $) and $ C' $ ($ s', t', u' $) with children $ x $ and $ x' $.
            If Bob places $ x $ below $ C' $, then Alice wins.
        }
        \label{fig:childBelowAunt}
    \end{figure}

    Finally, we consider vertices and edges outside the initial cliques and show that any two cliques that are copies of each other are laid out alternatingly and that each edge is assigned to the same queue as its copy.
    Recall that Bob cannot introduce new queues due to \cref{cond:differentColors}.
    Again, creating sufficiently many copies of the current graph (with the same initial clique \cinit) yields two graphs $ H $ and $ H' $ such that for every edge $ e \in E(H) $, the corresponding edge in $ H' $ is assigned to the same queue as $ e $.
    For the vertex ordering, we continue the induction starting with the alternating initial cliques.
    For this, consider a clique $ C $ and its copy $ C' $ and add a child each (see \cref{fig:childBelowAunt}).
    The rightmost vertex of the cliques are denoted by $ u $ and $ u' $, their children by $ x $ and $ x' $, respectively.
    If Bobs places $ x $ below $ C' $, then $ ux $ nests below all edges between $ u' $ and its parent clique.
    However, $ x $ and $ x' $ have incident edges in the same $ k $ queues, and thus $ ux $ forms a rainbow with an edge in the same queue.
    As $ u'x' $ is a copy of $ ux $, the two edges are assigned to the same queue and thus may not nest.
    Hence, Bob has to place the children in the same order as their parent cliques.

\end{document}